\documentclass[12pt]{amsart}
\usepackage{fullpage}
\usepackage{amsmath, amssymb, amsthm}
\usepackage{hyperref}

\title{\textbf{Lacunary recurrences and 2-adic properties of Eisenstein series}}
\author{Liubomir Chiriac}
\address{Portland State University, 1855 SW Broadway, Portland, OR 97201}
\email{chiriac@pdx.edu}

\author{Andrei Jorza}
\address{University of Notre Dame, 275 Hurley Hall, Notre Dame, IN 46556}\email{ajorza@nd.edu}
\date{}

\newcommand{\SL}{\operatorname{SL}}

\newtheorem{theorem}{Theorem}
\newtheorem{proposition}[theorem]{Proposition}

\newtheorem{lemma}[theorem]{Lemma}

\begin{document}


\maketitle

\begin{abstract}
We study the rational coefficients that arise when the Eisenstein series $G_k$ is expressed as a polynomial in $G_4$ and $G_6$, proving a conjecture that gives an exact formula for their minimal 2-adic valuation in terms of the binary expansion of the weight. The proof uses lacunary recurrences for Eisenstein series and yields refined information about the first valuation levels. As an application, we prove irreducibility results for
Faber polynomials associated to dyadic linear combinations of powers of
Eisenstein series. 
\end{abstract}

\section{Introduction}
The location of the zeros of Eisenstein series, and their algebraic
complexity, has long been a topic of interest. A classical theorem of Rankin and Swinnerton-Dyer \cite{rankin-swinnerton-dyer} states that
all the zeros of the Eisenstein series
\[E_k=\frac{1}{2}\sum_{(a,b)=1}\frac{1}{(az+b)^k}\]
lie on the unit arc forming part of the boundary of the standard fundamental
domain for $\SL_2(\mathbb{Z})$. Subsequent work has studied not only
the location of these zeros, but also the algebraic degree of their images
under the $j$-invariant. This problem is naturally approached via the Faber
polynomial associated to $E_k$. Indeed, writing $\Delta=(E_4^3-E_6^2)/1728$, $j=E_4^3/\Delta$, and $k=12D+k'$ with $k'\in\{0,4,6,8,10,14\}$, the quotient 
\[
\frac{E_k}{\Delta^D E_{k'}}
\]
is a polynomial in $j$, and its roots are precisely the $j$-values of the non-elliptic zeros of $E_k$. Thus, questions about the algebraic complexity of the zeros lead to the arithmetic of the coefficients $w(a,k) \in \mathbb{Q}$ appearing when the normalized Eisenstein series $G_k=2 \zeta(k)E_k$ is expanded in terms of the generators $G_4$ and $G_6$, that is   
$$G_k = \sum_{4a+6b=k} w(a,k) G_4^a G_6^b.$$
In this paper, we determine the minimal 2-adic valuation $v_2$ among the coefficients $w(a,k)$. More precisely, if $s(n)$ is the number of ones in the binary expansion of $n$, we prove the following theorem.

\begin{theorem} \label{MainResult}
For every even $k \ge 4$ we have
\[
\min_a  v_2 \left( w(a,k) \right) =
\begin{cases}
0 & \text{if $k$ is a power of 2},\\
s(k)-2 & \text{if $k$ is not a power of 2}.
\end{cases}
\]
\end{theorem}
\noindent The formula above was originally conjectured by Gonz\'alez \cite[Conj. 8]{Gonz}, who verified it for all weights $k\le 3500$. Using a recurrence due to Popa \cite{Popa}, Gonz\'alez also established the universal lower bound \[\min_a  v_2 \left( w(a,k) \right)\ge 0.\] While that non-negativity result served as a critical technical component in \cite{Gonz}, the exact minimum and its precise dependence on the binary expansion of the weight remained elusive. Our proof of the theorem relies on a combinatorial identity first discovered by Romik \cite{Romik}, alongside subsequent generalizations by Mertens and Rolen \cite{MerRol}. Interestingly, Mertens and Rolen describe these recurrences as being lacunary because they relate higher-weight Eisenstein series to sparse products of lower-weight series. Such inherent sparsity, combined with a high degree of coefficient symmetry, allows us to identify terms achieving the minimum and ensure a predictable valuation behavior. 

Beyond the exact determination of the minimum valuation, our argument also yields finer information about the coefficients lying at the first few valuation levels for some weights. These sharper estimates have arithmetic
consequences for the Faber polynomials introduced above. As a first application, we treat the modular form $E_k^2+E_{2k}$. Its zeros were shown in \cite{klangwang} to lie on the arc of the unit circle, so they satisfy an analogue of the Rankin--Swinnerton-Dyer phenomenon. The irreducibility of its Faber polynomial provides information about the algebraic complexity of the associated $j$-values.

\begin{theorem} \label{FaberIrred}
Let $k=12D$ with $D$ a power of 2. Then the Faber polynomial associated to $E_k^2+E_{2k}$ is irreducible over $\mathbb{Q}$. 
\end{theorem}

It is worth pointing out that, under the same assumption on $k$, Gonz\'alez \cite[Thm. 1]{Gonz} showed that the Faber polynomial associated to $E_k$ is irreducible; the key input is that $w(0,k)$ is the unique coefficient of minimal 2-adic valuation for such $k$. In our setting, however, uniqueness of the minimal valuation is not sufficient; the square $E_k^2$ introduces convolution terms, and controlling the intermediate coefficients requires information about the second valuation level. The stronger coefficient estimates developed in Section~\ref{attain-min} provide precisely the needed extra divisibility. 

In fact, Theorem~\ref{FaberIrred} is the first nontrivial instance of a more
general dyadic behavior.  The same second-valuation estimates that handle
the convolution in \(E_k^2\) also apply to higher dyadic powers and certain
integral linear combinations of them.  This leads to the following
generalization.

\begin{theorem}\label{t:linearcombo}
Let $k=12D$ where $D$ is a power of 2. Suppose \[0\leq a_1 <a_2<\cdots <a_t\leq \ell\] and $m_1,\ldots, m_t$ are integers such that $m_1$ and $\frac{1}{2}\sum_{j=1}^t m_j$ are odd integers. Then the Faber polynomial associated to the weight $2^\ell k$ modular form \[\sum_{j=1}^t m_j E_{2^{a_j}k}^{2^{\ell-a_j}}\] is irreducible over $\mathbb{Q}$.
\end{theorem}

\section{Powers of two}
We begin by addressing the case where the weight $k$ is a power of 2. By Gonz\'alez's non-negativity result \cite[Theorem 7]{Gonz}, we are guaranteed that $v_2(w(a,k)) \ge 0$ for all $k\ge 4$. Thus, it is enough to show that a valuation of 0 can be achieved. In fact, we will show that only one coefficient has this property. 

\begin{proposition} \label{Powers-of-2}
For every integer $m \ge 0$ and $k = 2^{m+2}$, we have $v_2(w(2^m, k)) = 0$ and $v_2(w(a, k)) > 0$ whenever $a \neq 2^m$. 
\end{proposition}

\begin{proof}

For any $x, y \in \mathbb{Q}$, we write $x \equiv y \pmod 2$ to mean $v_2(x - y) \ge 1$. We extend this notation coefficient-wise to polynomials in $G_4$ and $G_6$. 

We prove the statement by induction on $m$; the base $m=0$ is clear. Assume the claim holds for some $m-1 \ge 0$, i.e., $v_2(w(2^{m-1}, 2^{m+1})) = 0$, and all other coefficients $w(a, 2^{m+1})$ have a strictly positive 2-adic valuation. 

Recall the classical recurrence relation (see, for example, \cite[(59.6)]{Radem})
\begin{equation}  \label{classical_rec}
(n-3)(2n-1)(2n+1)G_{2n} = 3 \sum_{\substack{p+q=n \\ p,q \ge 2}} (2p-1)(2q-1)G_{2p}G_{2q}.
\end{equation}
Applied to $n=2^{m+1}$ it yields
\[ 
G_{2^{m+2}} \equiv  \sum_{\substack{p+q=2^{m+1} \\ p,q \ge 2}} G_{2p}G_{2q} 
 \equiv G_{2^{m+1}}^2+\sum_{\substack{p+q=2^{m+1} \\ 2\le p<q}} 2G_{2p}G_{2q} \equiv G_{2^{m+1}}^2 \pmod 2. 
\]
To extract the coefficient $w(a, 2^{m+2})$ of $G_4^a G_6^b$ we see that 
$$w(a, 2^{m+2}) \equiv \sum_{\substack{a_1+a_2=a \\ b_1+b_2=b}} w(a_1, 2^{m+1}) w(a_2, 2^{m+1}) \pmod 2,$$where $4a_i + 6b_i = 2^{m+1}$. Every off-diagonal pair $(a_1, b_1) \neq (a_2, b_2)$  occurs twice, so only a diagonal contribution can survive modulo 2. If $a$ is odd, no such diagonal contribution exists, so $w(a,2^{m+2}) \equiv 0\pmod 2$. If $a$ is even then $b$ is also even, and the only diagonal contribution is $(a_1,b_1)=(a_2,b_2)=(a/2,b/2)$. In that case \[w(a, 2^{m+2}) \equiv  w(a/2, 2^{m+1})^2 \pmod 2,\] and the conclusion follows from the inductive hypothesis. 
\end{proof}

Although the recurrence \eqref{classical_rec} efficiently isolates the pure $G_4$ coefficient when $k$ is a power of 2, extending this approach to general weights does not seem feasible. First, the presence of $(n-3)$ on the left-hand side makes it difficult to control the 2-adic valuation, particularly when $n$ is odd. More importantly, for general weights, the recurrence is prohibitively dense; because of the sheer number of terms in the summation, there is no clear mechanism to identify a specific $w$-coefficient that achieves the minimal valuation.

\section{Proving the lower bound} \label{proving-lower-bnd}
Having established the theorem for powers of 2, we now assume that $k \ge 6$ is not a power of 2. In this section, we prove the lower bound $v_2(w(a,k))\ge s(k)-2$ by induction on $k$. Thus, we assume that for every weight $k'<k$ the corresponding bound is known; namely
\[ 
v_2(w(a,k'))\ge
\begin{cases}
0 & \text{if $k'$ is a power of \(2\)},\\
s(k')-2 & \text{otherwise}.
\end{cases}
\]
We consider three cases. 

\subsection{Case 1: \texorpdfstring{$k \equiv 2 \pmod 6$}{k=2(mod6)}} \label{Case1:k=6n+2}

Let $k = 6n+2$. We make use of Romik's identity \cite[Theorem 1.6]{Romik} (see also \cite[Theorem 1]{MerRol})
\[
(6n+1) ((2n)!)^2 G_{6n+2} = (4n+1)! \sum_{j=1}^n \frac{\binom{2n}{2j-1}}{\binom{6n}{2n+2j-1}} G_{2n+2j} G_{4n-2j+2}.
\]
Setting
\[ c_j = \frac{(4n+1)!}{(6n+1)!} \frac{1}{(2n)!} \frac{(2n+2j-1)!(4n-2j+1)!}{(2j-1)!(2n-2j+1)!},\] we can write
\[ G_{6n+2}=\sum_{j=1}^n c_j G_{2n+2j}G_{4n-2j+2}.\] 

Since $v_2(m!) = m - s(m)$, we see that $v_2(c_j)$ is equal to
\begin{align*}
 & s(6n+1) + s(2n) + s(2j-1) + s(2n-2j+1) - s(4n+1) - s(2n+2j-1) - s(4n-2j+1) \notag \\
&=s(6n+2) -2 +v_2(6n+2) + s(2j-1) + s(2n-2j+1) - s(2n+2j-1) - s(4n-2j+1), \notag
\end{align*}
where we also used the fact that $s(m) -s(m-1)= 1-v_2(m)$ and $s(4m) = s(2m)$.

\bigskip

\noindent \textbf{Middle Term}. When $n$ is odd, a middle term arises for $j=(n+1)/2$, and
\[ v_2(c_{(n+1)/2})=s(6n+2)-2+v_2(6n+2)+2s(n)-2s(3n)=2s(n)-s(3n).\] By the inductive hypothesis, the contribution from $c_{(n+1)/2} G_{3n+1}^2$ is at least 
\[ 2s(n)-s(3n)+2(s(3n+1)-2)\ge s(k)-2,\] because $2s(n)-s(3n)=s(n)+s(2n)-s(n+2n)$ is precisely the number of carries when adding $n$ and $2n$ in binary, and $s(3n+1)=s(k)\ge 2$ as $k$ is not a power of 2 by assumption. 

In fact, a closer inspection reveals that the inequality is strict. To see that, assume by contradiction that equality takes place. Then $2s(n)=1+v_2(3n+1)$, and so $t=v_2(3n+1)$ must be odd. Also, $3n\equiv -1\pmod {2^t}$ implies that $n\pmod {2^t}$ is congruent to $\frac{2^{t+1}-1}{3}=1010\ldots 01_2$ with $t$ bits alternating 1 and 0; in particular, $s(n)\ge (t+1)/2$. The condition $2s(n)=1+t$ forces $n$ to be equal to $\frac{2^{t+1}-1}{3}$, which contradicts the fact that $t=v_2(3n+1)$.

\bigskip

\noindent Since $c_j=c_{n-j+1}$, we can pair symmetric summands. The pairs take the form  \[ 2c_j G_{2n+2j} G_{4n-2j+2}.\] Because both indices are between $2n+2$ and $4n$, they cannot be distinct powers of 2. If they were the same power of 2, their sum $(2n+2j)+(4n-2j+2)=k$ would also be a power of 2, which contradicts our initial assumption. 

\bigskip

\noindent \textbf{One power of 2.} Assume that for some $j$ we have that $2n+2j=2^r$ and $4n-2j+2$ is not a power of 2. Observe that $s(2n+2j-1)=r$ and $2(2j-1)+(2n-2j+1)=2^r-1$, which has $r$ ones in binary. It follows that
\[ r=s(2(2j-1))+s(2n-2j+1)=s(2j-1)+s(2n-2j+1).\] This yields
\[ v_2(c_j) = s(6n+2) -2 +v_2(6n+2) -s(4n-2j+1).\] Moreover, because $4n-2j+2<2^{r+1}$ and is not a power of 2, it must be true that $v_2(4n-2j+2)<r$. Thus \[ v_2(6n+2)=\min \left( v_2(2n+2j),v_2(4n-2j+2) \right)=v_2(4n-2j+2).\] Substituting this, we obtain 
\[ v_2(c_j) = s(6n+2) - 1 - s(4n-2j+2).\]

By the inductive hypothesis, the minimum 2-adic valuation of the $w$-coefficients occurring in the representation of $G_{2n+2j}$ is 0, while for $4n-2j+2$ this minimum is $s(4n-2j+2)-2$. In conclusion, all $w$-coefficients occurring in the decomposition for $2c_j G_{2n+2j} G_{4n-2j+2}$ have 2-adic valuation at least 
\[ 1+\left( s(6n+2) - 1 - s(4n-2j+2) \right)+ \left(s(4n-2j+2)-2\right)=s(k)-2. \]

\noindent \textbf{No powers of 2.} Suppose that neither $2n+2j$ nor $4n-2j+2$ is a power of 2. By the inductive hypothesis, the least 2-adic valuation among the coefficients appearing in the decomposition of $G_{2n+2j}$ is $s(2n+2j)-2$, while the corresponding minimum for $G_{4n-2j+2}$ is $s(4n-2j+2)-2$. Hence, every coefficient arising from $2c_j G_{2n+2j} G_{4n-2j+2}$ has valuation at least 
\begin{align*}
 1 & + v_2(c_j)+(s(2n+2j)-2)+(s(4n-2j+2)-2).\\
 &=s(k)-3+v_2(6n+2)-v_2(2n+2j)-v_2(4n-2j+2)+s(2j-1)+s(2n-2j+1).
\end{align*}
Note that $s(x)+s(y)\ge 1+v_2(N)$ whenever $N=2x+y+1$ is not a power of 2; indeed, $s(x)+s(y)=s(2x)+s(y)\ge s(2x+y)=s(N)+v_2(N)-1\ge 1+v_2(N)$. This implies that 
\[ s(2j-1)+s(2n-2j+1)\ge 1+\max \left( v_2(2n+2j),v_2(4n-2j+2) \right).\] Thus, the combined valuation from above is at least 
\[ s(k)-2+[v_2(6n+2)-\min\left( v_2(2n+2j),v_2(4n-2j+2) \right)]\ge s(k)-2.\] In particular, if the lower-bound is not strict then we must have
\[v_2(k)=\min\left( v_2(2n+2j),v_2(4n-2j+2) \right).\]

\subsection{Case 2: \texorpdfstring{$k \equiv 0 \pmod 6$}{k=0(mod6)}} \label{Case2:k=6n}

Let $k = 6n$. The base case $k=6$ is clear, since $w(0,6)=1$ and $s(6)-2=0$. For $n\ge 2$, we appeal to \cite[Theorem 2]{MerRol}
\begin{align*}
& \binom{6n+1}{2n} G_{6n}  = \notag \\ & \sum_{j=1}^n \left[ \binom{2n+2j-1}{2n}\binom{4n-2j-1}{2n} + 2\binom{2n+2j-1}{2n}\binom{4n-2j-1}{2n-2} \right] G_{2n+2j} G_{4n-2j}. 
\end{align*}

Letting $$a_j = \frac{\binom{2n+2j-1}{2n}\binom{4n-2j-1}{2n}}{\binom{6n+1}{2n}}, \quad b_j = \frac{2\binom{2n+2j-1}{2n}\binom{4n-2j-1}{2n-2}}{\binom{6n+1}{2n}},$$we can write$$G_{6n} = \sum_{j=1}^n (a_j + b_j) G_{2n+2j}G_{4n-2j}.$$ As before, we find that for $1\le j <n$
\[v_2(a_j)=s(6n) + s(2j-1) + s(2n-2j-1) - s(2n+2j-1) - s(4n-2j-1),\]and similarly, \[ v_2(b_j)=s(6n) + v_2(n) + s(2j-1) + s(2n-2j+1) - s(2n+2j-1) - s(4n-2j-1).\] Note that $v_2(b_j)-v_2(2a_j)=v_2(n)-v_2(n-j)$ and $v_2(b_{n-j})-v_2(2a_j)=v_2(n)-v_2(j).$ 

\bigskip

\noindent \textbf{Middle term}. When $n$ is even the term corresponding to $j=n/2$ is $(a_{n/2}+b_{n/2})G_{3n}^2$. Then \[v_2(a_{n/2}+b_{n/2})=v_2(a_{n/2})=2s(n)-s(3n)\ge 0,\] and by the inductive hypothesis the contribution from this middle term has valuation at least 
\[ 2s(n)-s(3n)+2(s(3n)-2)=s(3n)+2s(n)-4\ge s(6n)-2=s(k)-2,\] with equality exactly when $n$ is a power of 2. 

\bigskip

\noindent \textbf{Boundary term.} The term corresponding to $j=n$ is $b_nG_{4n}G_{2n}$, and $v_2(b_n)=s(6n)-s(n)$. If $n$ is a power of 2 then $s(6n)-s(n)>s(k)-2$, whereas if $n$ is not a power of 2 then \[(s(6n)-s(n))+(s(4n)-2)+(s(2n)-2)=s(6n)+s(n)-4\ge s(k)-2.\] Either way, the necessary lower bound is satisfied by the induction hypothesis. Note that this is the only case when both $2n+2j$ and $4n-2j$ can simultaneously be powers of 2. 

\bigskip

\noindent Excluding the indices $j\in \{n/2,n\}$, we can pair equal summands, and the pairs take the form \[ \left( 2a_j + b_j + b_{n-j} \right) G_{2n+2j} G_{4n-2j}.\] 

\noindent \textbf{One power of 2.} When $2n+2j=2^r$ and $4n-2j$ is not a power of 2, we find that \[ v_2(2a_j + b_j + b_{n-j})=v_2(b_j)=s(6n)+v_2(n)-s(4n-2j-1).\] With the contribution from the inductive hypothesis, the minimal 2-adic valuation of the $w$-coefficients coming from $(2a_j + b_j + b_{n-j}) G_{2n+2j} G_{4n-2j}$ is
\[s(6n)+v_2(n)-v_2(4n-2j)-1=s(k)-2.\] It is important to note that such an index $j$ exists if and only if $n$ is not a power of 2. 

\bigskip

\noindent \textbf{No powers of 2.} When neither $2n+2j$ nor $4n-2j$ is a power of 2, we distinguish three possibilities depending on how $v_2(n)$ and $v_2(j)$ relate. More precisely, 
\[ v_2(2a_j+b_j+b_{n-j})=
	\begin{cases}
		v_2(2a_j) & \text{ if } v_2(n)>v_2(j),\\
		v_2(b_j) & \text{ if } v_2(n)=v_2(j),\\
		v_2(b_{n-j}) & \text{ if } v_2(n)<v_2(j).
	\end{cases}
\]
If $v_2(n)>v_2(j)$, the minimal valuation coming from $(2a_j+b_j+b_{n-j})G_{2n+2j} G_{4n-2j}$ is 
\begin{align*}
& 1+s(6n)+s(2j-1) + s(2n-2j-1) - v_2(n+j) - v_2(2n-j)-4 \notag \\
&=s(6n)+s(2j-1) + s(2n-2j-1) -2v_2(j)-3 \notag \\
&=s(k)+s(j)+s(n-j)-3 > s(k)-2.
\end{align*} 
Similarly, the total lower bound simplifies to 
\[ s(k)+s(2j-1)+s(2n-2j+1)-v_2(n+j)-4,\] if $v_2(n) = v_2(j)$, and to 
\[ s(k)+s(2n-2j-1)+s(2j+1)-v_2(2n-j)-4,\] if $v_2(n) < v_2(j)$. In each of these possibilities we use the inequality: $s(x)+s(y)\ge 1+v_2(N)$ whenever $N=2x+y+1$ is not a power of 2. It yields, in both cases, that the total valuation is again at least $s(k)-2$. 

Finally, note that once the strict case $v_2(n)>v_2(j)$ is excluded, we necessarily have
\[ v_2(k)=\min \left( v_2(2n+2j), v_2(4n-2j) \right).\]


\subsection{Case 3: \texorpdfstring{$k \equiv 4 \pmod 6$}{k=4(mod6)}} \label{Case3:k=6n+4}

Let $k = 6n+4$. We now apply \cite[Theorem 3]{MerRol}
\begin{align*}
& \left[ \binom{6n+3}{2n+2} + 2\binom{6n+3}{2n} \right] G_{6n+4} = \notag \\
&\sum_{j=1}^{n+1} \left[ \binom{2n+2j-1}{2n}\binom{4n-2j+3}{2n} + 2\binom{2n+2j-1}{2n}\binom{4n-2j+3}{2n+2} \right] G_{2n+2j} G_{4n-2j+4}.
\end{align*}

The coefficient on the left-hand side simplifies to $L=\binom{6n+3}{2n} \frac{6n+5}{n+1}$, which has valuation  
\[ v_2(L) = s(n) + s(n+1) + 2 - s(6n+4) - v_2(6n+4).\]
On the right-hand side, we define
$$a_j = \binom{2n+2j-1}{2n}\binom{4n-2j+3}{2n} \text{ and } b_j = 2\binom{2n+2j-1}{2n}\binom{4n-2j+3}{2n+2}.$$ 

\bigskip

\noindent \textbf{Middle term}. A middle term exists when $n$ is even, and it is of the form $(a_{n/2+1} + b_{n/2+1})G_{3n+2}^2$ with $$v_2(a_{n/2+1} + b_{n/2+1}) = v_2(a_{n/2+1})=2\left( s(n)+s(n+1)-s(3n+1) \right).$$ Combining this with the inductive hypothesis for $G_{3n+2}^2$ and subtracting $v_2(L)$ yields
\[ (s(k)-2)+(s(n)+s(n+1)-v_2(3n+2)-1).\] Writing $n=2m$, the second term above becomes $2s(m)-v_2(3m+1)-1$. Put $t=v_2(3m+1)$, and note that $3m+1$ is not a power of $2$, since $k=4(3m+1)$ is not. If $t$ is odd, the alternating-binary argument used in the middle-term case for $k=6n+2$ gives $s(m)>(t+1)/2$, so $t\le 2s(m)-3$. If $t$ is even, the same argument applies with $\frac{2^{t}-1}{3}=0101\ldots01_2$ in place of $\frac{2^{t+1}-1}{3}$: now $m\equiv \frac{2^{t}-1}{3} \pmod{2^t}$, so $s(m)\ge t/2$, with equality only if $3m+1=2^t$, which is excluded. In both cases $v_2(3m+1)\le 2s(m)-2$. Hence, the middle-term contribution is $>s(k)-2$. 

\bigskip 

For $j\neq n/2+1$, pairing the symmetric terms yields a similar tripartite framework:$$(2a_j + b_j + b_{n+2-j}) G_{2n+2j} G_{4n-2j+4},$$ and we claim that 
\[ v_2( 2a_j + b_j + b_{n+2-j}) = v_2(a_j)+1-v_2(n+1).\] Indeed, expressing $b_j$ and $b_{n+2-j}$ in terms of $a_j$ it follows that
\[ b_j=a_j \frac{(2n-2j+2)(2n-2j+3)}{(n+1)(2n+1)} \text{ and }
b_{n+2-j}=a_j \frac{(2j-2)(2j-1)}{(n+1)(2n+1)},\] leading to 
\[ 2a_j+b_j+b_{n+2-j}=a_j \left( \frac{2(4n^2+8n+5-4nj-8j+4j^2)}{(n+1)(2n+1)} \right),\] which implies the claim. 

\bigskip 

\noindent \textbf{One power of 2.}  When $2n+2j$ is a power of 2, we find that
\[ v_2(a_j)+1-v_2(n+1)= 1 + s(n+1) + s(2n-2j+1) - s(4n-2j+3),\] which combined with the contributions from the inductive hypothesis gives
\[
V :=   s(n+1) + s(2n-2j+1) - v_2(4n-2j+4).
\]
Subtracting the valuation of $L$ yields
$$V-v_2(L) = s(6n+4) - 2 + \left[ s(2n-2j+1) - s(n) \right] - \left[ v_2(4n-2j+4) - v_2(6n+4) \right].$$ It is not hard to see that both of these brackets evaluate to 0, showing that $V-v_2(L) = s(k) - 2$ and also that $v_2(k)= \min \left( v_2(2n+2j), v_2(4n-2j+4) \right).$

\bigskip

\noindent \textbf{No powers of 2.} Suppose that neither $2n+2j$ nor $4n-2j+4$ is a power of 2. The lower bound for the contribution arising from $(2a_j + b_j + b_{n+2-j}) G_{2n+2j} G_{4n-2j+4}$ after subtracting $v_2(L)$ is
\[ V':=s(k)-4+v_2(k)+s(2j-1)+s(2n-2j+3)-v_2(2n+2j)-v_2(4n-2j+4).\] Clearly $s(2j-1)+s(2n-2j+3) = 2 + s(j-1)+s(n-j+1)$. As in the previous cases, we use that $s(x)+s(y)\ge 1+v_2(N)$ when $N=2x+y+1$ is not a power of 2. Setting $x=j-1$ and $y=n-j+1$ (and vice versa) yields
\[ s(j-1)+s(n-j+1)\ge \max(v_2(2n+2j),v_2(4n-2j+4)).\]
Hence $s(2j-1)+s(2n-2j+3)\ge 2+\max(v_2(2n+2j),v_2(4n-2j+4)),$ and so
\[ V'\ge s(k)-2+v_2(k)-\min(v_2(2n+2j),v_2(4n-2j+4)).\] In conclusion $V'\ge s(k)-2$, and if this lower bound is not strict then necessarily
\[ v_2(k)=\min(v_2(2n+2j),v_2(4n-2j+4)).\]

\section{Attaining the lower bound} \label{attain-min}

We now prove that the lower bounds established above can be attained. 

Denote by $\mathcal{L}_k$ the set of unordered pairs $\{A,B\}$ of even integers with $A+B=k$ and $k/3\le A,B\le 2k/3$, allowing $A=B$. For $k\ge 8$ every such $A$ and $B$ is at least $4$, and after pairing symmetric terms, each identity from the previous section reads
\[ G_k=\sum_{\{A,B\}\in \mathcal{L}_k} d_{A,B} G_A G_B,\] for a corresponding coefficient $d_{A,B}$ attached to the pair $\{A,B\}$.

Set
\[
\lambda(k)=
\begin{cases}
0, & k\text{ is a power of }2,\\
s(k)-2, & k\text{ is not a power of }2,
\end{cases}
\qquad
\mu(k)=
\begin{cases}
0, & k\text{ is a power of }2,\\
2^{v_2(k)-1}, & k\text{ is not a power of }2.
\end{cases}
\]

\begin{lemma} \label{witness-prep}
Let $k\ge 10$ be even, not a power of $2$, and not of the form $3\cdot 2^m$. 
	\begin{enumerate}
		\item Exactly one pair $\{P,Q\} \in \mathcal{L}_k$ contains a power of $2$. Only one of its entries, say $P$, is a power of 2, and $v_2(Q)=v_2(k)<v_2(P).$ Moreover $\mu(Q)=\mu(k)$ and 
		\[ v_2(d_{P,Q})+\lambda(Q)=\lambda(k).\]
		\item Every pair $\{A,B\}\in \mathcal{L}_k$ with neither entry a power of $2$ satisfies \[ v_2(d_{A,B})+\lambda(A)+\lambda(B)\ge \lambda(k),\] and moreover
			\begin{enumerate}
				\item if $v_2(A)\neq v_2(B)$, then $\mu(A)+\mu(B)>\mu(k)$; 
				\item if $v_2(A)= v_2(B)$, then $v_2(d_{A,B})+\lambda(A)+\lambda(B)> \lambda(k)$.  
			\end{enumerate}
	\end{enumerate}
\end{lemma}

\begin{proof}
	(1) Let $2^r$ be the largest power of $2$ with $2^r\le 2k/3$. Since $2^{r+1}>2k/3$, it is clear that $k/3<2^r$. Any other power of $2$ in $[k/3,2k/3]$ is at most $2^{r-1}$, and $2^{r-1}\ge k/3$ would force $2^r=2k/3$, i.e., $k=3\cdot 2^{r-1}$ which is not allowed. So $P=2^r$ is the only power of $2$ in $[k/3,2k/3]$, and in fact $k/3<P<2k/3$, since $k\neq 3\cdot 2^{r-1}$. Moreover, $P\ge 4$ because $k\ge 10$.
	
	If $Q=k-P$ then $k/3<Q<2k/3$, and $Q\neq P$ since $k$ is not a power of $2$. Thus $Q$ is not a power of 2, $\{P,Q\}$ is the only pair containing one. As $Q\neq 2^r$ and $Q<2k/3<2^{r+1}$, it follows that $v_2(Q)<r$ and so $v_2(k)=v_2(P+Q)=v_2(Q)<v_2(P)$. 
	
	Finally, $\{ P,Q\}$ is the summand treated under ``\textbf{One power of 2}'' in the previous section, and as shown there
	\[ v_2(d_{P,Q})+\lambda(Q)=s(k)-2.\]
	
	(2) The first claim is what Section~\ref{proving-lower-bnd} proves. 
	
	For the second claim, if $v_2(A)\neq v_2(B)$, we may assume that $v_2(A)<v_2(B)$. Then $v_2(k)=v_2(A)$, so $\mu(k)=\mu(A)$, while $\mu(B)\ge 1$. Hence, clearly $\mu(A)+\mu(B)>\mu(k)$. 
	
	Now suppose $v_2(A)=v_2(B)$, so that $v_2(k)>\min\left( v_2(A),v_2(B) \right)$. Note that the boundary pair $\{4n,2n\}$ from Section~\ref{Case2:k=6n} is excluded, since $v_2(4n)\neq v_2(2n)$. 
		\begin{itemize}
			\item If $A=B$, the summand is a middle term. In Section~\ref{Case1:k=6n+2} and Section~\ref{Case3:k=6n+4} it is shown that the bound for this term is strict. The bound in Section~\ref{Case2:k=6n} is strict as well, because $k/6$ is not a power of $2$. 
			\item If $A\neq B$, the summand is a paired term. The ``\textbf{No powers of 2}'' analysis in each subsection of Section~\ref{proving-lower-bnd} shows that the bound can fail to be strict only if $v_2(k)=\min(v_2(A),v_2(B))$, which is not the case here.
		\end{itemize}
\end{proof}

Whenever $4\mid k-6b$, we shall denote by $W_{k,b}$ the coefficient of  $G_4^{(k-6b)/4}G_6^b$ in $G_k$, i.e., 
\[ W_{k,b}:= w\left(\frac{k-6b}{4},k\right). \]
In particular, $ v_2(W_{k,b})\ge \lambda(k)$ by the preceding sections. 

\begin{proposition} \label{Min-Witness}
Let $k\ge 6$ be even and not a power of 2. 

	\begin{enumerate}
		\item[(a)] Suppose $k$ is not of the form $6n$ with $n$ a power of 2. Then $v_2(W_{k,\mu(k)})=\lambda(k), $ and if \(b<\mu(k)\) then $v_2(W_{k,b})>\lambda(k).$
		\item[(b)] Suppose $k=6n$, where $n$ is a power of 2. Then 
			$v_2(W_{k,n})=0$. If $n\ge2$, then moreover $v_2(W_{k,0})=1,$ and for $b\notin\{0,n\}$ we have $v_2(W_{k,b})>1$. 
	\end{enumerate}
\end{proposition}

\begin{proof}

We argue by induction on $k$. The base case $k=6$ is clear because $W_{6,1}=1$ is the only coefficient of $G_6$. Note that, by (b), the conclusion of (a) also holds when $k=6n$ with $n$ a power of $2$, since then $\mu(k)=n$ and $\lambda(k)=0$; we use this freely when applying the induction hypothesis. 

(a)  Note that $k\ge 10$ and that $k$ is not of the form $3\cdot 2^m$, so Lemma~\ref{witness-prep} applies. After pairing symmetric terms, there is a unique distinguished summand of the form $ dG_PG_Q,$ where $P$ is a power of $2$ and $Q=k-P$ with $ v_2(k)<v_2(P).$ Therefore $v_2(Q)=v_2(k)$ and $\mu(Q)=\mu(k)$. Moreover, the scalar $d$ satisfies $ v_2(d)+\lambda(Q)=\lambda(k). $ All of this follows from Lemma~\ref{witness-prep}(1). 

For every other summand $d'G_AG_B$, with $A+B=k$, neither $A$ nor $B$ is a power of $2$, and Lemma~\ref{witness-prep}(2) gives either
\[ v_2(d')+\lambda(A)+\lambda(B)>\lambda(k), \]
or else equality holds and $\mu(A)+\mu(B)>\mu(k)$ (equality rules out $v_2(A)=v_2(B)$).

Expanding
\[ dG_PG_Q=d \left( \sum_u W_{P,u} G_4^{(P-6u)/4} G_6^u \right) \left( \sum_v W_{Q,v} G_4^{(Q-6v)/4} G_6^v \right),\] 
we see that its contribution to \(W_{k,\mu(k)}\) is $d\sum_{u+v=\mu(k)}W_{P,u}W_{Q,v}. $ The term with $u=0$ and $v=\mu(k)=\mu(Q)$ is $dW_{P,0}W_{Q,\mu(Q)}$. Since $P$ is a power of 2, $W_{P,0}$ is a 2-adic unit. By the induction hypothesis, $v_2(W_{Q,\mu(Q)})=\lambda(Q).$ Therefore this term has valuation $v_2(d)+v_2(W_{P,0})+v_2(W_{Q,\mu(Q)}) = v_2(d)+\lambda(Q) = \lambda(k). $ Every other way to obtain total $G_6$-degree $\mu(k)$ has $u>0$ and $v=\mu(k)-u<\mu(k)=\mu(Q).$ Since $P$ is a power of 2, $W_{P,u}$ has positive valuation for $u>0$.  Also, by induction, $v_2(W_{Q,v})>\lambda(Q)$ whenever $v<\mu(Q)$.  Hence every term with  $u>0$ contributes with valuation strictly larger than $v_2(d)+ \lambda(Q)=\lambda(k).$ Thus the total contribution of the distinguished summand to $W_{k,\mu(k)}$ has valuation exactly $\lambda(k)$.

Now, consider a non-distinguished summand $d'G_AG_B$, with $A+B=k$. Its contribution to $W_{k,\mu(k)}$ is $d'\sum_{u+v=\mu(k)}W_{A,u}W_{B,v}$. If
$ v_2(d')+\lambda(A)+\lambda(B)>\lambda(k),$ then every individual term $d'W_{A,u}W_{B,v}$ has valuation \(>\lambda(k)\), because $v_2(W_{A,u})\ge \lambda(A)$, and $v_2(W_{B,v})\ge \lambda(B).$ Therefore this whole summand contributes to \(W_{k,\mu(k)}\) with valuation strictly larger than  $\lambda(k)$.

The remaining possibility is $v_2(d')+\lambda(A)+\lambda(B)=\lambda(k).$
By Lemma~\ref{witness-prep}(2): $\mu(A)+\mu(B)>\mu(k).$ For every pair $u,v$ with $u+v=\mu(k),$ we must have either
$ u<\mu(A)$  or $v<\mu(B).$ By the induction hypothesis, this implies either
$v_2(W_{A,u})>\lambda(A)$ or $v_2(W_{B,v})>\lambda(B).$ Thus
$ v_2(W_{A,u}W_{B,v})>\lambda(A)+\lambda(B),$
and hence $v_2(d'W_{A,u}W_{B,v})>\lambda(k).$ So this non-distinguished summand also contributes to $W_{k,\mu(k)}$ with valuation strictly larger than $\lambda(k)$. Therefore the only contribution to $W_{k,\mu(k)}$ with valuation exactly $\lambda(k)$ comes from the distinguished summand, and so $v_2(W_{k,\mu(k)})=\lambda(k).$

The same argument proves the second assertion.  Indeed, if $b<\mu(k)$, then in the distinguished summand every decomposition $u+v=b$ either has $u>0$, or has $u=0$ and $v=b<\mu(Q)$.  In both cases the contribution has valuation strictly larger than $\lambda(k)$.  For every non-distinguished summand, the same argument as above applies, since $ b<\mu(k)<\mu(A)+\mu(B) $ in the equality case.  Hence $v_2(W_{k,b})>\lambda(k) $ for all \(b<\mu(k)\).

\bigskip

(b) The case $k=6$ was settled above, and the case $k=12$ is immediate because $G_{12}=\frac{18}{143}G_4^3+\frac{25}{143}G_6^2$; so assume $n\ge 4$. We use the notation from the $k=6n$ identity. This time, the distinguished summand is the middle term $ dG_{3n}^2$, with $d=a_{n/2}+b_{n/2}$, and $v_2(d)=2s(n)-s(3n)=0.$ Also, note that $\mu(3n)=n/2$ and $\lambda(3n)=0.$

By induction, the coefficient $W_{3n,n/2}$ has valuation 0, while every coefficient $W_{3n,u}$ with $u<n/2$ has positive valuation. The contribution of the middle term to $W_{6n,n}$ contains the term $d\,W_{3n,n/2}^2,$ which has valuation 0. All other decompositions $u+v=n$ have at least one of $u,v$ smaller than $n/2$, and hence contribute with positive valuation.  Thus $v_2(W_{6n,n})=0.$

Next consider $W_{6n,0}$.  The middle term contributes $ d\,W_{3n,0}^2,$ which has valuation at least 2, since $v_2(W_{3n,0})>0$.  The boundary term is $b_nG_{4n}G_{2n},$ and, as computed in the $k=6n$ case, $v_2(b_n)=s(6n)-s(n)=1.$ Both $4n$ and $2n$ are powers of 2, so their pure $G_4$-coefficients are 2-adic units.  Therefore the boundary term contributes to $W_{6n,0}$ with valuation exactly 1. For the remaining paired terms $d_jG_{2n+2j}G_{4n-2j}$, with  $d_j=2a_j+b_j+b_{n-j}$ and $j\neq n/2,n$, the computations in the $k=6n$ section give $v_2(d_j)=1.$ Moreover, the two lower weights cannot both have $\lambda=0$.  Indeed, this would force both $j$ and $n-j$ to be powers of 2; hence $j=n/2$, the excluded middle case.  Therefore every coefficient coming from $G_{2n+2j}G_{4n-2j}$ has positive valuation, and so these paired terms contribute to \(W_{6n,0}\) with valuation at least \(2\). Hence $v_2(W_{6n,0})=1.$

Finally, suppose $b\notin\{0,n\}$. The boundary term contributes with valuation $>1$, since producing positive $G_6$-degree requires a positive $G_6$-degree coefficient from $G_{4n}$ or $G_{2n}$, and all such coefficients have positive valuation. The remaining paired terms again have valuation $1$ and product contribution of positive valuation, hence contribute with valuation $>1$. 

It remains to consider the middle term $dG_{3n}^2$. By induction applied to
$3n=6(n/2)$, we have $v_2(W_{3n,n/2})=0$, $v_2(W_{3n,0})=1,$ and $v_2(W_{3n,u})>1$ for every admissible $u\notin\{0,n/2\}$. Thus, for $0<b<n$, every decomposition $u+v=b$ gives valuation $>1$, except possibly when $b=n/2$. In that case, the two borderline terms occur symmetrically and
combine as $2dW_{3n,0}W_{3n,n/2}$, which has valuation at least $2$. Therefore, the middle term also contributes
with valuation $>1$. Hence $v_2(W_{6n,b})>1$ whenever $0<b<n$.

\end{proof}

\section{An application to Faber polynomials} \label{application}

Let $M_k$ be the space of modular forms of even weight $k\ge 4$ and level one. Write \[k=12D+k'\] for unique $D\ge 0$ and \[k'\in\{0,4,6,8,10,14\}.\] 

For every $f\in M_k$, we have that the quotient \[ \frac{f}{\Delta^D E_{k'}}\] is a polynomial in $j$; we refer to this polynomial as the Faber polynomial associated to $f$. Numerical evidence (see \cite{Gekeler}) suggests that the Faber polynomial associated to each $E_k$ is irreducible, and Gonz\'alez confirmed that observation when $k=12D$ and $D$ is a power of 2. The technical idea behind Gonz\'alez's proof is that the 2-adic valuation of $w(0,k)$ is exactly 0, while all other coefficients have strictly positive valuation. 

In this section we prove a similar phenomenon for the form $E_k^2 + E_{2k}$, whose zeros were studied in \cite{klangwang}. After squaring $E_k$, the coefficients of the resulting Faber polynomial involve convolutions of the coefficients coming from $E_k$, and Gonz\'alez's bound alone is insufficient. Instead, we make use of the sharper information provided by Proposition~\ref{Min-Witness}. 

We now recall the first irreducibility result stated in the introduction.

\medskip

\noindent\textbf{Theorem~\ref{FaberIrred}.}
\textit{Let $k=12D$ with $D$ a power of 2. Then the Faber polynomial associated to $E_k^2+E_{2k}$ is irreducible over $\mathbb{Q}$. }

\medskip

\begin{proof}
Substituting \[G_4=2\zeta(4)E_4=\frac{\pi^4}{45}E_4\] and \[G_6=2\zeta(6)E_6=\frac{2\pi^6}{945}E_6\] in the expression 
\[E_k=\frac{1}{2\zeta(k)} \sum_{a=0}^D w(3a,k)G_4^{3a} G_6^{2(D-a)}\] yields
\[ E_k=\frac{\pi^k}{2 \zeta(k)}\sum_{a=0}^D w(3a,k) \frac{2^{2(D-a)}}{45^{3a} ~ 945^{2(D-a)}} E_4^{3a} E_6^{2(D-a)}.\]
Next, we note that 
\[ \frac{E_4^{3a}E_6^{2(D-a)}}{\Delta^D}=j^a (j-1728)^{D-a}=\sum_{r=a}^D\binom{D-a}{D-r} (-1728)^{D-r}j^r.\] It follows that, for $0\le r<D$, the coefficient of $j^r$ in $E_k/\Delta^D$ is 
\[ t_{k,r} := \frac{\pi^k}{\zeta(k)} (-1)^{D-r} \sum_{a=0}^r w(3a,k) \frac{2^{8D-6r-2a-1}}{3^{3D+3r}5^{2D+a}7^{2D-2a}} \binom{D-a}{D-r}.
\] Since the Faber polynomial associated to $E_k$ is monic, we also set $t_{k,D}=1$. 

Recall that \[ 2 \zeta(k) =-\frac{(2 \pi i)^k}{k!} B_k,\] where $B_k$ is the $k$-th Bernoulli number. Since $k=12D$ with $D$ a power of 2, we have that $s(k)=2$, and hence $v_2(k!)=k-s(k)=k-2$. Moreover $v_2(B_k)=-1$, by the von Staudt--Clausen Theorem. Therefore $v_2(\pi^k/\zeta(k))=0$.   

From Proposition~\ref{Min-Witness}, we know that $v_2(w(0,k))=0$, $v_2(w(3D,k))=1$ and for all $a\notin\{0,3D\}$ we have $v_2(w(a,k))\ge 2$. This implies that \[v_2(t_{k,0}) = 8D-1.\]

Let $0<r<D$. The $a=0$ summand has valuation at least
\[ 8D-6r-1 \ge 8(D-r)+1.\]
For $a>0$, using $v_2(w(3a,k))\ge 2$ and $a\le r$, the corresponding
summand has valuation at least
\[ 2+8D-6r-2a-1\ge 8(D-r)+1. \]
Consequently, for all $0<r<D$ we have
\[ v_2(t_{k,r})\ge 8(D-r)+1.\]

The Faber polynomial associated to the modular form $E_k^2+E_{2k}\in M_{2k}$ has degree \[d:=\frac{2k}{12}=2D,\] and leading coefficient 2. For convenience, we normalize it so that it becomes monic, and let
\[ \frac{1}{2}\frac{E_k^2+E_{2k}}{\Delta^d}=j^d+\sum_{r=0}^{d-1} c_{k,r}j^r,\] Since $c_{k,0}=\frac{1}{2}\left( t_{2k,0}+t_{k,0}^2 \right)$, we find that \[v_2(c_{k,0})=16D-3=8d-3.\] 
Further, for $0<r<d$, we have
\[ c_{k,r}=\frac{1}{2} \left(t_{2k,r}+\sum_{u+v=r} t_{k,u}t_{k,v} \right).\] We consider three ranges.

\begin{enumerate}
\item If $0<r<D$, the endpoint terms in the summation are $t_{k,0}t_{k,r}+t_{k,r}t_{k,0}=2t_{k,0}t_{k,r}$, so they have valuation at least
\[ 1+(8D-1)+(8(D-r)+1)=8(d-r)+1.\] All other terms have valuation at least 
\[ (8(D-u)+1)+(8(D-r+u)+1)=8(d-r)+2.\] Since also $v_2(t_{2k,r})\ge 8(d-r)+1$, we infer that 
\[ v_2(c_{k,r})\ge 8(d-r).\] 
\item If $r=D$, the endpoint contribution is $2t_{k,0}t_{k,D}$, which has valuation \[1+(8D-1)+0=4d.\] The remaining terms have valuation at least $8D+2$, and $v_2(t_{2k,D})\ge 8D+1$. Therefore \[v_2(c_{k,D})\ge 4d-1.\] 
\item If $D<r<2D$, the endpoint terms are $t_{k,D}t_{k,r-D}+t_{k,r-D}t_{k,D}=2t_{k,D}t_{k,r-D}$. Since $0<r-D<D$, their valuation is at least 
\[ 1+0+(8(D-(r-D))+1)=8(d-r)+2.\] All other terms have $0<u,v<D$, and hence valuation at least 
\[ \left(8(D-u)+1 \right)+\left( 8(D-v)+1 \right)=8(d-r)+2.\]
Meanwhile, $v_2(t_{2k,r})\ge 8(d-r)+1$ so
\[ v_2(c_{k,r})\ge 8(d-r).\] 
\end{enumerate}

The above analysis shows that for all $0<r<d$ we have
\[ \frac{v_2(c_{k,r})}{d-r} > \frac{v_2(c_{k,0})}{d}.\] Moreover, since $D$ is a power of 2, so is $d$, and hence $v_2(c_{k,0})=8d-3$ is coprime to $d$. By Dumas's criterion \cite[Corollary 1.3]{Dumas}, the Faber polynomial associated to $E_{k}^2+E_{2k}$ is irreducible over $\mathbb{Q}$. 
\end{proof}

\bigskip

It is useful to reinterpret the last part of the proof in terms of Newton
polygons. Let
\[
P(x)=x^d+\sum_{r=0}^{d-1}c_rx^r\in \mathbb Q[x],
\]
and put $h=v_2(c_0)$. Suppose that for $0<r<d$ we have
\[ v_2(c_r)\ge \frac{h(d-r)}{d}\]
and that $\gcd(h,d)=1$. Then the 2-adic Newton polygon of $P$ consists of a
single segment joining $(0,h)$ to $(d,0)$. Indeed, since $h$ and $d$
are coprime, this segment contains no lattice points other than its endpoints, so the above inequalities force all intermediate points to lie strictly above the segment. Dumas's criterion then implies that $P$ is irreducible over $\mathbb Q_2$, and hence over $\mathbb Q$.

A similar Newton polygon argument applies to a broader class of integral linear combinations of Eisenstein forms.  The hypotheses below ensure two things: first, that the resulting Faber polynomial can be normalized by dividing by a number of 2-adic valuation exactly $1$, and second, that its constant coefficient has a unique term of minimal valuation.

\bigskip

\noindent\textbf{Theorem~\ref{t:linearcombo}.}
\textit{Let $k=12D$ where $D$ is a power of 2. Suppose \[0\leq a_1 <a_2<\cdots <a_t\leq \ell\] and $m_1,\ldots, m_t$ are integers such that $m_1$ and $\frac{1}{2}\sum_{j=1}^t m_j$ are odd integers. Then the Faber polynomial associated to the weight $2^\ell k$ modular form \[\sum_{j=1}^t m_j E_{2^{a_j}k}^{2^{\ell-a_j}}\] is irreducible over $\mathbb{Q}$. }

\begin{proof}
Let
\[ D_j=2^{a_j}D,\qquad N_j=2^{\ell-a_j},\qquad d=2^\ell D, \]
so that $N_jD_j=d$. Let
\[ T_j(x)=\frac{E_{2^{a_j}k}}{\Delta^{D_j}} =\sum_{u=0}^{D_j} t_{2^{a_j}k,u}x^u. \]
Then the Faber polynomial of
\[
f=\frac{1}{M}\sum_{j=1}^t m_jE_{2^{a_j}k}^{N_j},
\qquad M=\sum_{j=1}^t m_j,
\]
is
\[
F_f(x)=\frac{1}{M}\sum_{j=1}^t m_j T_j(x)^{N_j}
      =x^d+\sum_{r=0}^{d-1}c_rx^r.
\]
By assumption $M/2$ is odd, so $v_2(M)=1$.

From the proof of Theorem~\ref{FaberIrred}, applied to the weight
$2^{a_j}k=12D_j$, we have
\[ v_2(t_{2^{a_j}k,0})=8D_j-1, \]
\[ v_2(t_{2^{a_j}k,u})\ge 8(D_j-u)+1 \qquad (0<u<D_j),
\]
and
\[ t_{2^{a_j}k,D_j}=1. \]

We first compute the constant coefficient. We have
\[
c_0=\frac{1}{M}\sum_{j=1}^t m_jt_{2^{a_j}k,0}^{N_j}.
\]
For each \(j\),
\[
v_2\left(m_jt_{2^{a_j}k,0}^{N_j}\right)
\ge
N_j(8D_j-1)=8d-N_j.
\]
Since $m_1$ is odd, the \(j=1\) term has valuation exactly
\[
8d-N_1=8d-2^{\ell-a_1},
\]
while all other terms have strictly larger valuation. Hence
\[
v_2(c_0)=8d-N_1-v_2(M)=8d-N_1-1.
\]
Set
\[ h:=8d-N_1-1. \]
Since $m_1$ is odd and $M=\sum_{j=1}^tm_j$ is even, we must have $t\ge 2$. Hence $a_1<a_2\le \ell$, and $N_1=2^{\ell-a_1}$ is even. Therefore $h$ is odd, and $\gcd(h,d)=1$ because $d=2^{\ell}D$ is a power of $2$. It remains to prove that all intermediate coefficients lie on or above the
line joining \((0,h)\) and \((d,0)\). Equivalently, we need for $0<r<d$
\[ v_2(c_r)\ge \frac{h(d-r)}{d}.  \]

Fix $j$, and let $S_{j,r}$ be the coefficient of $x^r$ in
$T_j(x)^{N_j}$. It suffices to prove
\[ v_2(S_{j,r})-1\ge \frac{h(d-r)}{d}.  \]

Write $N=N_j=2^m$ and $D'=D_j$, so that $ND'=d$. Consider a contribution
to $S_{j,r}$ coming from a multiset of indices $u_1,\ldots,u_N$, $0\le u_i\le D'$,  with
\[ u_1+\cdots+u_N=r. \]
Let $A$ be the number of indices equal to $0$, let $B$ be the number
equal to $D'$, and let $C=N-A-B$ be the number of remaining indices. Denote
the remaining indices by $v_1,\ldots,v_C$, where $0<v_i<D'$.

Using the valuation estimates for the coefficients $t_{2^{a_j}k,u}$, such a
grouped contribution has valuation at least
\[ 
v_2\left(\binom{N}{A,B,C}\right) + A(8D'-1) + \sum_{i=1}^C \bigl(8(D'-v_i)+1\bigr).
\]
Since
\[
r=BD'+v_1+\cdots+v_C
\]
and $d=ND'$, this lower bound becomes
\[
v_2\left(\binom{N}{A,B,C}\right)+8(d-r)-A+C.
\]

We claim that
\[
v_2\left(\binom{N}{A,B,C}\right)+C\ge \frac{B+C}{N}.
\]
Indeed, the right-hand side is at most $1$. If $C>0$, then the left-hand
side is at least $1$. If $C=0$, then the left-hand side is
$v_2\left(\binom{N}{A}\right)$. Since $N$ is a power of $2$, this is positive unless
$A=0$ or $A=N$. The case $A=N$ gives $r=0$, while the case $A=0$
gives $r=d$, both excluded. Therefore the claim holds for $0<r<d$.

Now
\[
r\le (B+C)D',
\]
so
\[
\frac{r}{d}\le \frac{B+C}{N}.
\]
Using the previous claim, we get
\[
v_2\left(\binom{N}{A,B,C}\right)\textsc{}+8(d-r)-A+C
\ge
8(d-r)-N+(N+1)\frac{r}{d}.
\]
Thus
\[
v_2(S_{j,r})\ge
8(d-r)-N_j+(N_j+1)\frac{r}{d}
\ge 8(d-r)-N_1+(N_1+1)\frac{r}{d}.
\]
After subtracting $v_2(M)=1$, this gives
\[
v_2(c_r)\ge
8(d-r)-N_1-1+(N_1+1)\frac{r}{d}.
\]
But the right-hand side is exactly
\[
\frac{(8d-N_1-1)(d-r)}{d}
=
\frac{h(d-r)}{d}.
\]

Thus every point $(r,v_2(c_r))$, $0<r<d$, lies above the line segment
joining $(0,h)$ and $(d,0)$, and by Dumas's criterion, $F_f$ is irreducible over $\mathbb Q$.
\end{proof}

\bigskip

\subsection*{Acknowledgements}
We thank the referee for many helpful and detailed comments. 

\bigskip

\noindent \textbf{Data Availability}. No datasets were generated or analysed during the current study.

\bigskip

\noindent \textbf{Conflict of interest}. The authors declare no conflict of interest.

\bigskip

\noindent \textbf{Funding}. The authors declare that they have not received funding for this work.

\bigskip

\bibliographystyle{alpha}

\bibliography{ref}

\end{document}